\documentclass[11pt,reqno]{amsart}
\usepackage{graphicx}
\usepackage{amssymb,amsmath}
\usepackage{amsthm}
\usepackage{color,graphicx}
\usepackage{hyperref}
\usepackage{color}
\usepackage{epstopdf}

\newcommand{\be}{\begin{equation}}

\newcommand{\ee}{\end{equation}}

\newcommand{\Hess}{{\rm\,Hess\,}}

\newcommand{\cn}{{\rm \,cn}}
\newcommand{\sn}{{\rm \,sn}}
\newcommand{\dn}{{\rm \,dn}}
\newcommand{\sech}{{\rm \,sech}}
\newcommand{\R}{{\mathbb R}}

\newcommand{\ve}{{\varepsilon}}

\numberwithin{equation}{section}
\numberwithin{figure}{section}

\newtheorem{theorem}{Theorem}[section]
\newtheorem{proposition}[theorem]{Proposition}
\newtheorem{remark}[theorem]{Remark}
\newtheorem{lemma}[theorem]{Lemma}
\newtheorem{corollary}[theorem]{Corollary}
\newtheorem{definition}[theorem]{Definition}

\begin{document}
\vglue-1cm \hskip1cm
\title[Stability of Periodic-Wave Solutions ]{Nonlinear Stability of Periodic-Wave Solutions for Systems of Dispersive Equations}

\begin{center}



\maketitle

{\bf Fabr\'icio Crist\'ofani }

{IMECC-UNICAMP\\
	Rua S\'ergio Buarque de Holanda, 651, CEP 13083-859, Campinas, SP,
	Brazil.}\\
{ fabriciocristofani@gmail.com}

\vspace{3mm}

{\bf Ademir Pastor}

{IMECC-UNICAMP\\
	Rua S\'ergio Buarque de Holanda, 651, CEP 13083-859, Campinas, SP,
	Brazil.}\\
{apastor@ime.unicamp.br}




\end{center}

\begin{abstract}
We prove the orbital stability of periodic traveling-wave solutions for systems of dispersive equations with coupled nonlinear terms. Our method is basically developed under two assumptions: one concerning the spectrum of the linearized operator around the traveling wave and another one concerning the existence of a conserved quantity with suitable properties. The method can be applied to several systems such as the Liu-Kubota-Ko system, the  modified KdV system and a log-KdV type system. 
\end{abstract}

\section{Introduction}

Results of orbital stability of periodic
traveling waves related to the following vector-valued system of dispersive equations
\be\label{gKDVsystem}
\partial_t u+\partial_x\left(\nabla R(u)-\mathcal{M}u \right)=0
\ee
will be shown in this manuscript. Here, $u$ is the vector $(u_1,u_2)$ where $u_1$ and $u_2$ are spatially periodic real functions depending on the real variables $(x,t)\in\R\times\R$. The nonlinear part is represented by the term $\nabla R(u)$, where $R:\R^2\to\R$ is a sufficiently smooth function. Also, the dispersion operator $\mathcal{M}$ is such that
\be\label{operatorm}
\mathcal{M}=\begin{pmatrix}
\mathcal{M}_{11} & \mathcal{M}_{12}\\
\mathcal{M}_{21} & \mathcal{M}_{22}	
\end{pmatrix}
\ee
with $\mathcal{M}_{ij}$ given as Fourier multipliers by
\be\label{symbol1}
\widehat{\mathcal{M}_{ij}g}(\kappa)=m_{ij}(\kappa)\widehat{g}(\kappa), \quad \kappa\in\mathbb{Z}.
\ee
The two-by-two matrix $m(\kappa)=(m_{ij}(\kappa))$ then  represents the symbol of $\mathcal{M}$.
Throughout the paper we assume that $m(\kappa)$  is symmetric (for each $\kappa\in\mathbb{Z}$) with $m_{ij}(\kappa)$  even and continuous functions such that, for any vector $w$ in the unity sphere of $\R^2$,
\begin{equation}\label{A1A2}
c_1|\kappa|^{s}\leq \langle m(\kappa)w,w\rangle_{\R^2}\leq c_2|\kappa|^{s}, \quad s>0, 
\end{equation}
for all $\kappa\in\mathbb{Z}$ and for some constants $c_i>0$, $i=1,2$. Here $\langle \cdot,\cdot\rangle_{\R^2}$ denotes the usual inner product in $\R^2$.

System \eqref{gKDVsystem} encompasses several important dispersive systems  arising  in physical phenomena. More specifically, system \eqref{gKDVsystem}  arises as model for nonlinear waves in a number of situations. We can cite as an example the Liu, Kubota and Ko system \cite{LKK} which models the interaction between a disturbance located at an upper pycnocline and another disturbance located
at a lower pycnocline in a three-layer fluid. Other examples are the surface water models put forward in \cite{bona4,bona5}, where the authors consider a formal derivation
of the four-parameter system (particular cases of \eqref{gKDVsystem}) in order to approximate the motion of small-amplitude long waves on the surface of an ideal fluid under the force of gravity and in situations where the motion is sensibly two dimensional.
The Gear-Grimshaw system is also a particular case of \eqref{gKDVsystem} and it was derived in \cite{GG} to model the strong interaction of two long internal gravity waves in a stratified fluid, where the two waves are assumed to correspond to different modes of the linearized equations of motion. For the precise description os these systems, see Section \ref{applications}.

Traveling waves are solutions of \eqref{gKDVsystem}  having the form $u_1(x,t)=\phi_1(x-c t)$ and $u_2(x,t)=\phi_2(x-ct)$, where $c\in\mathbb{R}$ indicates the wave speed and
$\phi_1$ and $\phi_2$ are smooth real-valued functions. Their existence implies a suitable balance between the effects of the nonlinearity and the frequency dispersion.  By replacing this form of solutions in \eqref{gKDVsystem} we see that $\phi=\left(\phi_1,\phi_2\right)$ must satisfy the nonlinear
(pseudo) differential system of equations
\be\label{travwave}
	\mathcal{M}\phi+c\phi-\nabla R(\phi)+A=0,
\ee
 where $A=(A_1,A_2)$ with $A_1$ and $A_2$ constants of integration.

The existence and orbital stability of \textit{solitary} traveling waves (this means that $\phi$ together with its all derivatives vanish at infinity) for several particular cases of system \eqref{gKDVsystem} have been considered in the literature, which we refrain from list them all here. However, our work is mainly inspired in \cite{bona3}, where the authors studied \eqref{gKDVsystem} with 
\begin{equation}\label{partM}
\mathcal{M}=
\begin{pmatrix}
-\partial_x^2 & 0 \\
0 & -\partial_x^2
\end{pmatrix}
\end{equation}
and $R$ a polynomial of degree 3.
 The authors obtained a stability criteria for proportional solitary waves of the form $(\psi,\mu\psi)$, $\mu \in \mathbb{R}$. More precisely, they proved that solitary waves of the form
$$
(\psi,\mu\psi)=(\mu_1,\mu_2)3\omega \sech^2\left(\frac{\sqrt{\omega}}{2}(x-\omega t)\right), \qquad \mu=\frac{\mu_2}{\mu_1},
$$
are orbitally stable in the energy space provided $\det\mathcal{H}(\mu_1,\mu_2)<1/2$, where 
$\mathcal{H}(\mu_1,\mu_2)$ is a suitable matrix depending on the coefficients of $R$.
Other special cases of \eqref{gKDVsystem} was studied by Albert and Linares in \cite{AL}. The authors proved the existence and stability of localized solutions representing coupled solitary waves traveling at a common speed. The developed theory  was applied to the models of Gear-Grimshaw and Liu, Kubota and Ko (\cite{GG}, \cite{LKK}).   In section \ref{applications}  we will establish the orbital stability of traveling waves for these models   in the periodic context.

 To describe our results, we assume that \eqref{gKDVsystem} conserves (at least formally) the energy
\begin{equation}\label{conser1}
E(u)=\int_{0}^{L}\left\{ \frac{1}{2}\langle \mathcal{M}u,u\rangle_{\R^2}-R(u)\right\}dx,
\end{equation}
 the momentum
\begin{equation}\label{conser2}
F(u)=\frac{1}{2}\int_{0}^{L}\langle u,u\rangle_{\R^2}dx,
\end{equation}
and the mass
\begin{equation}\label{conser3}
M(u)=\int_{0}^{L}\langle u,\overrightarrow{1}\rangle_{\R^2}dx,
\end{equation}
where $\overrightarrow{1}=(1,1)$ and  $L>0$ is the minimal period of $u$. In view of \eqref{A1A2} the function space where the energy is well defined, the so called \textit{energy space}, is the standard Sobolev space $H_{per}^{\frac{s}{2}}([0,L])\times H_{per}^{\frac{s}{2}}([0,L])$.

Now, in view of the conserved quantities $(\ref{conser1})$, $(\ref{conser2})$ and $(\ref{conser3})$ we may define the augmented Lyapunov functional
\begin{equation}\label{lyafun}
G(u)=E(u)+cF(u)+\int_{0}^{L}\langle u,A\rangle_{\R^2}dx,
\end{equation}
and the linearized operator around the wave $\phi$,
\begin{equation}\label{operatorL}
\mathcal{L}:=G''(\phi) 
=
\mathcal{M}+
c \,{\rm Id_2}- 
\Hess R(\phi)
\end{equation}
where $G''(\phi)$ stands for the Fr\'echet derivative of $G$ at $\phi$, $\Hess R$ denotes the Hessian matrix of $R$, and ${\rm Id_2}$ is the two-by-two identity matrix.
 Note that,  $\mathcal{L}$ is a well defined self-adjoint operator  in $L_{per}^2([0,L])\times L_{per}^2([0,L])$ with domain $H_{per}^s([0,L])\times H_{per}^s([0,L])$. In order to simplify notation, the  norm and inner product in $L_{per}^2([0,L])\times L_{per}^2([0,L])$ will be denoted by  $||\cdot||$ and $\langle\cdot,\cdot\rangle$. 

Before stating our stability result, we make precise the notion of orbital stability. For functions $u$ and $v$ in the energy space $X= H_{per}^\frac{s}{2}([0,L])\times H_{per}^\frac{s}{2}([0,L])$ we consider $\rho$ as the ``distance'' between $u$ and $v$  defined  by
\begin{equation*}
\rho(u,v)=\inf_{y\in\mathbb{R}}||u-v(\cdot+y)||_{X}.
\end{equation*}
It is to be clear that the distance between $u$ and $v$ is measured trough the distance between $u$ and the orbit of $v$, generated under the action of translations. 

\begin{definition}\label{defstab}
	Let $\phi$ be an $L$-periodic solution of \eqref{travwave}.
	We say  $\phi$ is orbitally stable in $X$, by the periodic flow of \eqref{gKDVsystem},  if for any $\ve>0$ there exists $\delta>0$ such that for any $u_0\in X$ satisfying $\|u_0-\phi\|_X<\delta$, the solution $u(t)$ of \eqref{gKDVsystem} with initial data $u_0$ exists globally and satisfies
	$$
	\rho(u(t),\phi)<\ve,
	$$
	for all $t\geq0$.
\end{definition}

The notion of orbital stability prescribes the existence of global solutions. Since questions of (local and) global well-posedness is out of the scope of this paper, we will assume  the periodic Cauchy problem associated with (\ref{gKDVsystem}) is globally well-posed in $X$. We point out, however, that several theories related to the initial-value problem posed on $\mathbb{R}$ or in the periodic setting have been developed recently (see, for instance, \cite{AC,ACW, absta, oh}).

Our primary goal in this paper is to establish a criterion for the orbital stability for coupled dispersive systems based on the works \cite{ANP}, \cite{NP1}  and \cite{Stuart}. More specifically, we present a method to obtain the orbital stability of periodic traveling waves for dispersive systems by assuming the following assumptions:

\vspace{0.3cm}

\noindent 	$(H_1)$ There exists an $L$-periodic solution of (\ref{travwave}), say,  $\phi \in C^{\infty}_{per}([0,L])\times C^{\infty}_{per}([0,L])$, with minimal period $L>0$. Moreover, the self-adjoint operator $\mathcal{L}$ has only one negative eigenvalue, which is simple, and zero is a simple eigenvalue with associated eigenfunction $\phi'=\left(\phi_1',\phi_2'\right)$.

\vspace{0.2cm}

\noindent	$(H_2)$ There exists $\Phi\in X=H_{per}^\frac{s}{2}([0,L])\times H_{per}^\frac{s}{2}([0,L])$ such that
\begin{equation}\label{vak}
	I:=\langle\mathcal{L}\Phi,\Phi\rangle<0 \quad\mbox{and}\quad \langle\mathcal{L}\Phi,\phi'\rangle=0, 
\end{equation}
 Moreover, there exists a conserved functional $Q=Q(u)$  such that $Q'(\phi)=\mathcal{L}\Phi$.

\vspace{0.3cm}

We point out that we are not assuming the existence of a smooth curve (or surface) of periodic  traveling waves. The first assumption in \eqref{vak} appears as a substitute of the standard Vakhitov-Kolokolov type condition, which is crucial, for instance in the seminal works of Bona, Souganidis, and Strauss \cite{bona2} and Grillakis, Shatah, and Strauss \cite{grillakis1}.

Using these two assumptions our main result reads as follows.

\begin{theorem}\label{teoest} If assumptions $(H_1)$ and $(H_2)$ hold then $\phi$ is orbitally stable in $X$ by the periodic flow of $(\ref{gKDVsystem})$.
\end{theorem}

To prove Theorem \ref{teoest}, we follow the strategy in \cite{ANP}, \cite{CNP2}, \cite{NP1}  and \cite{Stuart}. 
Note that  solutions of \eqref{travwave} are critical points of the functional \eqref{lyafun}. Thus, as usual, it is expected that the stability of such critical points may be determined by studying the spectral properties of $\mathcal{L}$. 
The fundamental idea proposed in \cite{Stuart} is that instead of $G$, one uses the functional
$$V(u):=G(u)-G(\phi)+N(Q(u)-Q(\phi))^2,$$ 
as a Lyapunov function, where $N>0$ is a suitable constant to be determined later and $Q$ is defined in $(H_2)$. This function has the advantage that its coercivity can be proved in a simpler way (see Lemma \ref{lemma2}).

Our paper is organized as follows. In next section we will adapt the theory mentioned above in order to prove  Theorem \ref{teoest}. Section \ref{applications} is devoted to some applications of our developments. In particular, we show the orbital stability of periodic waves for the Liu-Kubota-Ko system and for coupled systems of KdV, mKdV, and Log-KdV equations.

\section{Proof of Theorem \ref{teoest}}\label{OSPW}

In this section, we basically rewrite the arguments in \cite{ANP}, \cite{CNP2}, \cite{NP1}  and \cite{Stuart}  to establish the theory of orbital stability for coupled dispersive systems. The next proposition is a classical result inserted in \cite{grillakis1} adapted to our case.

\begin{proposition}\label{operatorposit}
Suppose that assumptions $(H_1)$ and $(H_2)$ holds. Then, there exists a constant $\tau>0$ such that
$$\langle\mathcal{L}v,v\rangle\geq \tau||v||_{X}^2,$$
for all $v\in \Upsilon\cap \{\phi'\}^\perp$, where $\{\phi'\}^{\perp}:=\{u\in X ; \langle\phi',u\rangle=0\}$ and $\Upsilon:= \{u\in X; \, \langle\mathcal{L}\Phi,u\rangle=0\}$.
\end{proposition}
\begin{proof}
By using $(H_1)$ and the spectral decomposition theorem we can write
\be\label{decomp}
L_{per}^2([0,L])\times L_{per}^2([0,L])=[\chi]\oplus [\phi']\oplus P,
\ee
where $\chi$ is a normalized eigenfunction  associated to the unique negative eigenvalue  of $\mathcal{L}$, that is, $||\chi||=1$ and $\mathcal{L}\chi=-\sigma^2\chi$, $\sigma\neq0$. In addition $P$ denotes the positive subspace of $\mathcal{L}$, that is,
 $$\langle \mathcal{L}p,p\rangle\geq \tau_1||p||^2,\ \ \ \mbox{for all}\ p\in X\cap P,$$
for some positive constant $\tau_1$ .

Next, if $\Phi$ is defined in $(H_2)$, from $(\ref{decomp})$, we write
$$
\Phi=a_0\chi+b_0\phi'+p_0,\ \ \ \ \ a_0,b_0\in\mathbb{R},
$$
with $p_0\in X\cap P$. Thus, from $(H_2)$, 
\be\label{a1}
\langle \mathcal{L} p_0,p_0\rangle=\langle\mathcal{L}(\Phi-a_0\chi-b_0\phi'),\Phi-a_0\chi-b_0\phi'\rangle
=I+a_0^2\sigma^2<a_0^2\sigma^2.
\ee
Now, we take $v\in \Upsilon\cap \{\phi'\}^\perp$ with $||v||=1$. So we can write $v=a_1\chi+p_1$, where $p_1\in X\cap P$. Consequently,
\begin{equation}\label{a2}
0=\langle\mathcal{L}\Phi,v\rangle=\langle -a_0\sigma^2\chi +\mathcal{L}p_0,a_1\chi+p_1\rangle
=-a_0a_1\sigma^2+\langle\mathcal{L}p_0,p_1\rangle.
\end{equation}
From \eqref{a1} and \eqref{a2} we then deduce
$$
\langle \mathcal{L}v,v\rangle=-a_1^2\sigma^2+\langle \mathcal{L}p_1,p_1\rangle\geq -a_1^2\sigma^2+\frac{\langle \mathcal{L}p_0,p_1\rangle^2}{\langle \mathcal{L}p_0,p_2\rangle}>-a_1^2\sigma^2+\frac{(a_0a_1\sigma^2)^2}{a_0\sigma^2}=0.
$$
By using a contradiction argument (see \cite[Lemma 3.5]{ACN}) it is possible to verify that there exists a constant $\tau_2>0$ such that
\begin{equation}\label{ltau2}
\langle\mathcal{L}v,v\rangle\geq \tau_2||v||^2
\end{equation}
for all $v\in \Upsilon\cap \{\phi'\}^\perp$.

Finally, by considering $\Lambda$ as a particular case of \eqref{operatorm} such that $m_{ii}(\kappa)=|k|$, $i=1,2$, and $m_{12}(\kappa)=m_{21}(\kappa)=0$, we obtain, from \eqref{A1A2} and \eqref{ltau2},
\begin{equation}\label{ltau}
\left(1-c_3b\|q(\phi)\|_{Y}\right)\|v\|^2+b\|\Lambda^{s/2}v\|^2\leq \left(\frac{1}{\tau_2}+c_3b\right)\langle\mathcal{L}v,v\rangle,
\end{equation}
where $q(\phi)=c \,{\rm Id_2}-\Hess R(\phi)$, $c_3>0$ is a constant, $b>0$ is an arbitrary number chosen such that $\left(1-c_3b\|q(\phi)\|_{Y}\right)>0$  and $Y$ represents the space $L^{\infty}_{per}([0,L])\times L^{\infty}_{per}([0,L])$. Inequality \eqref{ltau} allows us to conclude the desired.
\end{proof}

With the above proposition in hand we  establish the following.

\begin{lemma}\label{lemma1}
Under assumptions of Proposition \ref{operatorposit}, there exist $N>0$ and $\widetilde{\tau}>0$ such that
$$\langle\mathcal{L}v,v\rangle +2N\langle \mathcal{L}\Phi,v\rangle^{2}\geq \widetilde{\tau}||v||_{X}^2,$$
for all $v\in \left\{\phi'\right\}^{\perp}$.
\end{lemma}
\begin{proof}
Given any $v\in \left\{\phi'\right\}^{\perp}$, we set
	$$z=v-\zeta w,$$
	where $w=\frac{Q'(\phi)}{||Q'(\phi)||}$ and $\zeta=\langle v,w\rangle$. Thus, $z\in\Upsilon\cap \left\{\phi'\right\}^{\perp}$ and therefore using Proposition \ref{operatorposit} we are able to establish that  
	\begin{equation}\label{eq01}
	\langle\mathcal{L}v,v\rangle \geq \zeta^2\langle\mathcal{L}w,w\rangle + 2\zeta\langle\mathcal{L}w,z\rangle +\tau||z||_X^2.
	\end{equation}

	Now, using Cauchy-Schwartz and Young's inequalities, we obtain
	\begin{equation}\label{eq02}
	|2\zeta\langle\mathcal{L}w,z\rangle| \leq \frac{\tau}{2}||z||_X^2 + \frac{2\zeta^2}{\tau}||\mathcal{L}w||_X^2.
	\end{equation}
	Furthermore, choosing $N=N(\phi)>0$ sufficiently large such that
	\begin{equation}\label{eq03}
	\langle\mathcal{L}w,w\rangle -\frac{2}{\tau}||\mathcal{L}w||_X^2 +2N||Q'(\phi)||^2\geq \frac{\tau}{2}\|w\|_{X}^2,
	\end{equation}
	we obtain, using (\ref{eq01}), (\ref{eq02}) and (\ref{eq03}),
	\[
	\begin{split}
	\langle\mathcal{L}v,v\rangle +2N\langle Q'(\phi),v\rangle^{2}&=\langle\mathcal{L}v,v\rangle + 2N\zeta^2||Q'(\phi)||^2  \\
	&\geq\frac{\tau}{2}(\zeta^2\|w\|^2_X+||z||_X^2)  \\
	&\geq \widetilde{\tau}||v||_{X}^2. 
	\end{split}
	\]
	The proof is thus completed.
\end{proof}

We define now the functional $V:X\rightarrow\R$ as
\begin{equation}\label{functionalV}
V(u)=G(u)-G(\phi)+N(Q(u)-Q(\phi))^2,
\end{equation}
where $G$ is the augmented functional defined in (\ref{lyafun}) and $N>0$ is the constant obtained in the previous lemma. It is easy to see from $(\ref{travwave})$ that $V(\phi)=0$ and $V'(\phi)=0$. Moreover, we have the following lemma.

\begin{lemma}\label{lemma2}
	There exist $\alpha>0$ and $D>0$ such that 
	$$V(u)\geq D\rho(u,\phi)^2$$
	for all $u\in U_{\alpha}:=\{u\in X;\ \rho(u,\phi) < \alpha\}$.
\end{lemma}
\begin{proof}
See Lemma 2.6 in \cite{CNP2}.	
\end{proof}

Finally, we are able to prove our main result.

\begin{proof}[Proof of Theorem \ref{teoest}]
Let $\varepsilon>0$ be given, which without loss of generality we assume $\varepsilon<\alpha/2$ where $\alpha$ is the constant given in Lemma \ref{lemma2}.. We need to show that there exists $\delta>0$ such that if $||u_0-\phi||_X<\delta$ then $\rho(u(t),\phi)<\varepsilon$, for all $t\geq0$.
 By using the continuity of $V$  at $\phi$ we may find $\delta\in (0,\alpha)$ such that if $||u_0-\phi||_X<\delta$  then
\be\label{estepsilon}
V(u_0)=V(u_0)-V(\phi)<D\varepsilon^2,
\ee
where $D>0$ is the constant in Lemma \ref{lemma2}. 

We will show that this $\delta$ is enough to our purpose. Indeed, since the function $t\mapsto\rho(u(t),\phi)$ is continuous and
$$
\rho(u_0,\phi)\leq \|u_0-\phi\|_X<\delta<\alpha,
$$
there exists $T>0$ such that 
\be\label{subalpha}
\rho(u(t),\phi)<\alpha,\ \ \  \mbox{for all}\ t\in [0,T),
\ee
which means $u(t)\in U_{\alpha}$, for all $t\in[0,T)$. Taking into account \eqref{estepsilon}, Lemma \ref{lemma2} and the fact that  $V(u(t))=V(u_0)$ for all $t\geq0$, we have
\be\label{estepsilon1}
\rho(u(t),\phi)<\varepsilon,\ \ \ \ \ \mbox{for all}\ t\in[0,T).
\ee

 Next, we prove that \eqref{subalpha} holds, for all $t\in [0,+\infty)$, which implies that \eqref{estepsilon1} also holds for all $t\in [0,+\infty)$ and yields the orbital stability. In fact, let  $T_1>0$ be the supremum of the values of $T>0$ for which $(\ref{subalpha})$ holds. We claim that  $T_1=+\infty$. On the contrary, by recalling that $\varepsilon<\alpha/2$ we obtain, from $(\ref{estepsilon1})$,
$$
\rho(u(t),\phi)<\frac{\alpha}{2}, \ \ \ \ \ \mbox{for all}\ t\in[0,T_1).
$$
By using again the continuity of the function $t\mapsto\rho(u(t),\phi)$, we may find $T_2>0$ such that
$\rho(u(t),\phi)<\frac{3}{4}\alpha<\alpha$, for $t\in [0,T_1+T_2)$, contradicting the maximality of $T_1$. This contradiction shows that, $T_1=+\infty$ and the theorem  is established.
\end{proof}

\section{Applications}\label{applications}

Here, we will apply the results developed in Section \ref{OSPW} in order to obtain the orbital stability of periodic waves for some systems arising in the current literature. The various applications in this section illustrate the effectiveness and versatility of our method.

\subsection{Coupled KdV System}

We start by considering the following coupled KdV system
\be\label{KDV}
\begin{cases}
	\partial_tu_{1} +\partial_x\big(B_1u_1^2+B_2u_1u_2+B_3u_2^2+\partial_x^2u_{1}\big)=0,\\
	\partial_tu_{2}+\partial_x\big(\frac{1}{2}B_2u_1^2+2B_3u_1u_2+B_4u_2^2+\partial_x^2u_{1}\big)=0,
\end{cases}
\ee
where $B_1,B_2,B_3$, and $B_4$ are real constants. Such systems appear as models describing the propagation of nonlinear waves, for instance, as the ones described by Gear-Grimshaw \cite{GG} and Majda-Biello \cite{maj}.
It is easy to see that \eqref{KDV} writes as \eqref{gKDVsystem} with $\mathcal{M}$ given in \eqref{partM} and
$$
R(u_1,u_2)=\frac{1}{3}B_1u_1^3+\frac{1}{2}B_2u_1^2u_2+B_3u_1u_2^2+\frac{1}{3}B_4u_2^3.
$$
In particular, \eqref{A1A2} holds with $s=2$. So, the energy space here is $X=H_{per}^1([0,L])\times H_{per}^1([0,L])$.

Let us search for proportional a periodic traveling wave $\phi=(\varphi(x-ct),\mu\varphi(x-ct))$, where $\mu$ is a real parameter. In this case, \eqref{travwave} reads as
\be\label{odekdvsystem}
\begin{cases}
-(\varphi''-c\varphi)-(B_1+B_2\mu+B_3\mu^2)\varphi^2+A_1=0,\\
-\mu(\varphi''-c\varphi)-(\frac{1}{2}B_2+2B_3\mu+B_4\mu^2)\varphi^2+A_2=0,
\end{cases}
\ee where $A_1$ and $A_2$ are constants of integration. By assuming that the cubic equation 
\begin{equation}\label{firstrelation}
\mu\left(B_1+B_2\mu+B_3\mu^2\right)=\frac{1}{2}B_2+2B_3\mu+B_4\mu^2
\end{equation}
has a real solution $\mu$ and taking $\mu A_1=A_2$ we see that system \eqref{odekdvsystem} reduces to a single equation. In particular, if we consider the change of variable $\varphi=\frac{1}{\xi}\widetilde{\varphi}$, with $\xi=2(B_1+B_2\mu+B_3\mu^2)\neq 0$, we obtain that $\widetilde{\varphi}$ must satisfy the ordinary differential equation
\begin{equation}\label{odekdv}
-\widetilde{\varphi}''+c\widetilde{\varphi}-\frac{1}{2}\widetilde{\varphi}^2+\widetilde{A}=0,
\end{equation}
 where $\xi A_1=\widetilde{A}$. In what follows we assume $\widetilde{A}=0$. Consequently, for any $L>0$, equation \eqref{odekdv} has  a positive $L$-periodic solution (see \cite[page 1144]{AN}) with the explicit form
 \begin{equation}\label{solkdv}
 \widetilde{\varphi}:=\widetilde{\varphi}_c(x)=\frac{16K^2}{L^2}\left(\sqrt{1-k^2+k^4}+1-2k^2\right)+\frac{48K^2k^2}{L^2}\cn^2\left(\frac{2K}{L}x,k\right),
 \end{equation}
with $c=c(k)=\frac{16K^2}{L^2}\left(\sqrt{1-k^2+k^4}\right)$. Here, $\cn$ represents the Jacobi elliptic function of  cnoidal type and $K=K(k)$ is the complete elliptic integral of the first kind. The parameter  $k\in(0,1)$ is the elliptic modulus (see, for instance, \cite{bird}).

Next, we will obtain the spectral properties related to the linearized operator
\begin{equation} \label{operatorkdv}
\begin{split}
\mathcal{L}&=
\left(-\partial_{x}^2+c\right)
\begin{pmatrix}
1 & 0 \\
0 & 1
\end{pmatrix}
- \varphi
\begin{pmatrix}
2B_1+\mu B_2 & B_2+2\mu B_3 \\
B_2+2\mu B_3 & 2B_3+2\mu B_4
\end{pmatrix}
 \\
&=\left(-\partial_{x}^2+c\right)
\begin{pmatrix}
1 & 0  \\
0 & 1
\end{pmatrix}
- \widetilde{\varphi}D
\end{split}
\end{equation}
where $D$ is the matrix 
$$D:=\frac{1}{\xi}
\begin{pmatrix}
2B_1+\mu B_2 & B_2+2\mu B_3 \\
B_2+2\mu B_3 & 2B_3+2\mu B_4 
\end{pmatrix}
$$

In order to obtain the spectral properties of $\mathcal{L}$ we will diagonalize it. 
The next result helps us to diagonalize the matrix $D$.

\begin{lemma}\label{eigenvalue}
The matrix $D$ has two real eigenvalues: $\lambda_1=1$ and $\lambda_2=\det D$.
\end{lemma}
\begin{proof}
First, note that $\det(D-\lambda I)=0$ if, and only if
\begin{equation}\label{a3}
\frac{1}{\xi^2}\det
\left( \begin{array}{cc}
2B_1+\mu B_2-\xi\lambda & B_2+2\mu B_3 \\
B_2+2\mu B_3 & 2B_3+2\mu B_4-\xi\lambda 
\end{array} \right)=0. 
\end{equation}
By multiplying the second column by $\mu$, adding the result to the first one and using \eqref{firstrelation} we see that \eqref{a3} is equivalent to
\begin{equation}
\frac{1}{\xi}(1-\lambda)\det
\left( \begin{array}{cc}
1 & B_2+2\mu B_3 \\
\mu & 2B_3+2\mu B_4-\xi\lambda 
\end{array} \right)=0. \nonumber
\end{equation}
 which in turn is equivalent to
$$(1-\lambda)(\det D-\lambda)=0.$$
Therefore, $\lambda_1=1$ and $\lambda_2=\det D$ solve $\det(D-\lambda I)=0$.
\end{proof}

Let $O$ be the orthogonal matrix whose columns are normalized eigenvectors of $D$ associated to the eigenvalues $1$ and $\det D$. It is easy to see that
\begin{equation}
O\mathcal{L}O^{-1}=
 \begin{pmatrix}
\mathcal{L}_1 & 0 \\
0 & \mathcal{L}_2
\end{pmatrix} 
\end{equation}
with
$$\mathcal{L}_1:=-\partial_{x}^2+c-\widetilde{\varphi} \quad \mbox{and} \quad \mathcal{L}_2:=-\partial_{x}^2+c-\lambda_2\widetilde{\varphi}. $$

With this information in hand we may establish the following result.
\begin{theorem}\label{teoLkdv}
Let $\varphi=\frac{1}{\xi}\widetilde{\varphi}$, where $\xi=2(B_1+B_2\mu+B_3\mu^2)$ and $\widetilde{\varphi}$ is defined in \eqref{solkdv}. If $\det D<\frac{1}{2}$, then assumption $(H_1)$ holds for $\mathcal{L}$ defined in \eqref{operatorkdv}.
\end{theorem}
\begin{proof}
	First note that the point spectrum of $\mathcal{L}$ coincides that of $O\mathcal{L}O^{-1}$. Thus it suffices to get the spectral properties of the diagonal operator $O\mathcal{L}O^{-1}$. 
 It is known in the literature that  $\mathcal{L}_1$ has only one negative eigenvalue, which is simple, and zero is a simple eigenvalue whose eigenfunction is $\widetilde\varphi'$ (see, for instance, \cite{AN} for details). On the other hand, note if $\mathcal{L}_3:=-\partial_{x}^2+c-\frac{1}{2}\widetilde{\varphi}$, then $\mathcal{L}_3\widetilde{\varphi}=0$, by \eqref{odekdv}. Since $\widetilde{\varphi}$ has no zeros in $[0,L]$, it follows from the oscillation theory (see, for instance, \cite[Theorem 3.1.2]{Eastham}) that $0$ is the first eigenvalue of $\mathcal{L}_3$. In consequence of the comparison theorem (see, for instance \cite[Theorem 2.2.2]{Eastham}) we conclude that $\mathcal{L}_2$ is a positive operator. Therefore, the proof is completed.
\end{proof}

In order to ensure the orbital stability of $\phi$ in $X=H_{per}^1[0,L]\times H_{per}^1[0,L]$ we now have to prove that $(H_2)$ holds. Indeed, Natali and Angulo showed in \cite[pages 1144 and 1145]{AN} that $c\in\left(\frac{4\pi^2}{L^2}, \infty\right)\mapsto \widetilde{\varphi}_c\in H_{per}^1[0,L]$ is a smooth curve of positive cnoidal waves solutions for \eqref{odekdv} (with $\widetilde{A}=0$), where $k:=k(c)$ is a strictly increasing smooth function. Moreover, they showed that $\frac{\partial}{\partial c}\|\widetilde{\varphi}\|^2_{L^2_{per}}$ is positive for all $c\in\left(\frac{4\pi^2}{L^2}, \infty\right)$. Thus, by setting $\Phi=\left(\frac{\partial}{\partial c}\varphi,\mu\frac{\partial}{\partial c}\varphi\right)\in X$, taking the derivative with respect to $c$ in \eqref{odekdvsystem} it is easily seen that $\mathcal{L}\Phi=-(\varphi,\mu\varphi)$. Hence, the functional $Q$ in $(H_2)$ may be taken to be $-F$, with $F$ as in \eqref{conser2}. In addition,
\begin{equation}
I=\langle\mathcal{L}\Phi,\Phi\rangle =-\frac{\left(\mu^2+1\right)}{2}\frac{\partial}{\partial c}\|\varphi\|^2_{L^2_{per}}
=-\frac{\left(\mu^2+1\right)}{2|\xi|}\frac{\partial}{\partial c}\|\tilde{\varphi}\|^2_{L^2_{per}} <0\nonumber
\end{equation}
 This implies that $(H_2)$ holds. 
 
As an application of Theorem \ref{teoest} we just have proved the following.

\begin{theorem}
Let $L>0$ be fixed. Suppose that \eqref{firstrelation} has a real solution $\mu$ and $\det D<\frac{1}{2}$. Let $\phi=(\varphi,\mu\varphi)=\frac{1}{\xi}(\widetilde{\varphi},\mu\widetilde{\varphi})$, with $\widetilde{\varphi}$ given in \eqref{solkdv}. Then $\phi$  is orbitally stable in $X=H_{per}^1([0,L])\times H_{per}^1([0,L])$ by the periodic flow of \eqref{KDV}. 
\end{theorem}

\subsection{Coupled Modified KdV System}

This subsection is devoted to show the orbital stability of periodic waves for the following coupled system of modified KdV equations
\be\label{mKDV}
\begin{cases}
	\partial_tu_{1} +\partial_x\left(D_1u_1^3+D_2u_1^2u_2+D_3u_1u_2^2+D_4u_2^3+\partial_x^2u_{1}\right)=0,\\
\partial_t	u_{2}+\left(\frac{1}{3}D_2u_1^3+D_3u_1^2u_2+3D_4u_1u_2^2+D_5u_2^3+\partial_x^2u_{2}\right)=0,
\end{cases}
\ee
where $D_1,D_2,\ldots,D_5$ are real coefficients. Systems of modified KdV equation have appeared, for instance, in the inverse scattering theory (see \cite{Ablo}). If $B_3=1$ and the other coefficients are all zero, the system was studied for instance in \cite{Alarcon}, \cite{bhata}, and \cite{Cor}). In particular, existence and stability/instability of solitary waves were shown. 

Here, it is easy to see that \eqref{mKDV} writes as  \eqref{gKDVsystem} with $\mathcal{M}$ defined in \eqref{partM} and
$$
R(u_1,u_2)=\frac{1}{4}D_1u_1^4+\frac{1}{3}u_1^3u_2+\frac{1}{2}D_3u_1^2u_2^2+D_4u_1u_2^3+\frac{1}{4}D_5u_2^4.
$$
As in the last subsection, we look for a proportional periodic wave $\phi=(\varphi(x-ct),\mu\varphi(x-ct))$, where $\mu$ is a real parameter. If we substitute $\phi$ into \eqref{mKDV} we obtain the pair of equations (which is equivalent to \eqref{travwave})
\be\label{odemkdvsystem}
\begin{cases}
	-(\varphi''-c\varphi)-(D_1+D_2\mu+D_3\mu^2+D_4\mu^3)\varphi^3+A_1=0,\\
	-\mu(\varphi''-c\varphi)-(\frac{1}{3}D_2+D_3\mu+3D_4\mu^2+D_5\mu^3)\varphi^3+A_2=0,
\end{cases}
\ee where $A_1$ and $A_2$ are constants of integration. Next we assume that the cubic equation
\begin{equation}\label{firstrelationm}
	\mu\left(D_1+D_2\mu+D_3\mu^2+D_4\mu^3\right)=\frac{1}{3}D_2+D_3\mu+3D_4\mu^2+D_5\mu^3 
\end{equation}
has a real solution $\mu$ and $A_2$ comes into the form $\mu A_1$. Assuming further that $\zeta:=3(D_1+D_2\mu+D_3\mu^2+D_4\mu^3)> 0$ we may consider
 the scalar change $\varphi=\left(\frac{6}{\zeta}\right)^{\frac{1}{2}}\widetilde{\varphi}$ to see that \eqref{odemkdvsystem} writes as the single equation
\begin{equation}\label{odemkdv}
	-\widetilde{\varphi}''+c\widetilde{\varphi}-2\widetilde{\varphi}^3+\widetilde{A}=0, \quad \mbox{with} \quad \displaystyle \widetilde{A}=\left(\frac{\zeta}{6}\right)^{\frac{1}{2}}A_1.
\end{equation}
For any $L>0$, equation \eqref{odemkdv} admits a smooth curve of positive $L$-periodic solutions (see \cite[page 1258]{PA}) in the explicit form
\begin{equation}\label{solmkdv}
	\widetilde{\varphi}:=\widetilde{\varphi}_k(x)=\frac{4K}{\sqrt{2}Lg(k)}\left(\frac{\dn^2\left(\frac{2K}{L}x,k\right)}{1+\gamma^2\sn^2\left(\frac{2K}{L}x,k\right)}\right),
\end{equation}
where $\gamma^2=\sqrt{k^4-k^2+1}+k^2-1$, $g(k)=\sqrt{\sqrt{k^4-k^2+1}-k^2+\frac{1}{2}}$, and $\dn$ represents the Jacobi elliptic function of dnoidal type. In addition, the parameters $c$ and $\widetilde{A}$ may be written in terms of elliptic modulus as 
\begin{equation*}
c=c(k)=\frac{16K}{L^2}\sqrt{k^4-k^2+1} 
\end{equation*}
and
\begin{equation*}
\widetilde{A}=\widetilde{A}(k)=-\frac{32K^3}{3\sqrt{3}L^3}\left(\sqrt{k^4-k^2+1}-2k^2+1\right)\sqrt{2\sqrt{k^4-k^2+1}+2k^2-1}.
\end{equation*}

Once the periodic solutions of \eqref{mKDV} are obtained, the next step is to obtain the spectral properties of the linearized operator  
\begin{equation} \label{operatormkdv}
\begin{split}
	\mathcal{L}&= 
	\left(-\partial_{x}^2+c\right)
	\begin{pmatrix}
		1 & 0 \\
		0 & 1
	\end{pmatrix} - \varphi^2
	 \begin{pmatrix} 
		3D_1+2\mu D_2+\mu^2D_3 & D_2+2\mu D_3+3\mu^2D_4 \\
		D_2+2\mu D_3+3\mu^2D_4 & D_3+6\mu D_4+3\mu^2D_5 
	\end{pmatrix}\\
	&=\left(-\partial_{x}^2+c\right)
	\begin{pmatrix}
	1 & 0 \\
	0 & 1
	\end{pmatrix}-6\,\widetilde{\varphi}^2H,
	\end{split}
\end{equation}
where $H$ is the matrix
$$H=\frac{1}{\zeta}
 \begin{pmatrix}
	3D_1+2\mu D_2+\mu^2D_3 & D_2+2\mu D_3+3\mu^2D_4 \\
D_2+2\mu D_3+3\mu^2D_4 & D_3+6\mu D_4+3\mu^2D_5 
\end{pmatrix} .$$

By using the same arguments as those in the proof of Lemma \ref{eigenvalue} we see that the eigenvalues of $H$ are  $\lambda_1=1$ and $\lambda_2=\det H$. Hence,  there exists an orthogonal matrix $O$ such that
\begin{equation*}
	O\mathcal{L}O^{-1}=
	\begin{pmatrix}
		\mathcal{L}_1 & 0 \\
		0 & \mathcal{L}_2
	\end{pmatrix},
\end{equation*}
with
$$\mathcal{L}_1:=-\partial_{x}^2+c-6\widetilde{\varphi}^2 \quad \mbox{and} \quad \mathcal{L}_2:=-\partial_{x}^2+c-6\lambda_2\widetilde{\varphi}^2, $$

Thus, in order to obtain the spectrum of $\mathcal{L}$ it is sufficient to obtain the spectrum of $	O\mathcal{L}O^{-1}$.
\begin{theorem}
Let $\varphi=\frac{1}{\zeta}\widetilde{\varphi}$, where $\zeta=3(D_1+D_2\mu+D_3\mu^2+D_4\mu^3)> 0$ and $\widetilde{\varphi}$ is defined in \eqref{solmkdv}. If $\det H<\frac{1}{3}$, then assumption $(H_1)$ holds for $\mathcal{L}$ defined in \eqref{operatormkdv}.
\end{theorem}
\begin{proof}
The proof is similar to that of Theorem \ref{teoLkdv}. Indeed, de Andrade and Pastor showed in \cite[page 1259]{PA} that $\mathcal{L}_1$ has only one negative eigenvalue, which is simple, and zero is a simple eigenvalue whose eigenfunction is $\widetilde\varphi'$. Also, since $\widetilde{\varphi}$ is positive, we may define the operator $\mathcal{L}_3:=-\partial_{x}^2+c-2\widetilde{\varphi}^2+\widetilde{A}/\widetilde{\varphi}$. From \eqref{odemkdv} we see that $\mathcal{L}_3\widetilde{\varphi}=0$, which implies that 0 is the first eigenvalue of $\mathcal{L}_3$. Since $\widetilde{A}<0$, if $\det H<\frac{1}{3}$, we have
$$
-2\widetilde{\varphi}^2+\widetilde{A}/\widetilde{\varphi}<-2\widetilde{\varphi}^2< -6\lambda_2\widetilde{\varphi}^2.
$$
The comparison theorem then yields that $\mathcal{L}_2$ is a positive operator and the proof of the theorem is completed.
\end{proof}

It remains tho show that $(H_2)$ occurs. Here we will explore the fact that $k\mapsto\widetilde{\varphi}$ is a smooth curve with respect to $k\in(0,1)$. Let $\Phi=(\frac{\partial}{\partial k}\varphi,\mu\frac{\partial}{\partial k}\varphi)$ and 
$$
Q(u)=-\left(\frac{1}{2}\frac{\partial c}{\partial k}\int_{0}^{L}(u_1^2+ u_2^2)dx+\frac{\partial A_1}{\partial k}\int_{0}^{L}(u_1+\mu u_2)dx\right).
$$ 
It follows readily from \eqref{odemkdvsystem} that
\begin{equation}
	\mathcal{L}\Phi =- \begin{pmatrix}
	\frac{\partial c}{\partial k}\varphi+\frac{\partial A_1}{\partial k} 
	\vspace{0.2cm}
	\\
	\mu\left(\frac{\partial c}{\partial k}\varphi+\frac{\partial A_1}{\partial k}\right) 
	\end{pmatrix} =Q'(\phi)  \nonumber
\end{equation}
 and, clearly, $\langle	\mathcal{L}\Phi, \phi'\rangle=0$. Finally,
 \[
\begin{split}
I&=\langle\mathcal{L}\Phi,\Phi\rangle \nonumber \\ &=-\left(\mu^2+1\right)\left\langle\left(\frac{\partial c}{\partial k}\varphi+\frac{\partial A_1}{\partial k}\right),\frac{\partial}{\partial k}\varphi\right\rangle \nonumber \\
&=-\frac{6\left(\mu^2+1\right)}{\zeta}\left\langle\left(\frac{\partial c}{\partial k}\widetilde{\varphi}+\frac{\partial \widetilde{A}}{\partial k}\right),\frac{\partial}{\partial k}\widetilde{\varphi}\right\rangle \nonumber \\
&=:\frac{6\left(\mu^2+1\right)}{\zeta}\widetilde{I} 
\end{split}
\]

Since $\zeta$ is positive, we will have $I<0$ provided $\widetilde{I}<0$. However,
 de Andrade and  Pastor \cite[page 1262]{PA} proved that, at least numerically, $\widetilde{I}<0$ for all $k\in (0,1).$ As a consequence of Theorem \ref{teoest} we just have established the following.

\begin{theorem}
Let $L>0$ be fixed. Suppose that \eqref{firstrelationm} has a real solution $\mu$ and $\det D<\frac{1}{3}$. Let $\phi=(\varphi,\mu\varphi)=\left(\frac{6}{\zeta}\right)^{\frac{1}{2}}(\widetilde{\varphi},\mu\widetilde{\varphi})$, with $\widetilde{\varphi}$ given in \eqref{solmkdv}. Then $\phi$  is orbitally stable in $X=H_{per}^1([0,L])\times H_{per}^1([0,L])$ by the periodic flow of \eqref{mKDV}. 
\end{theorem}

\subsection{Log-KdV System}

Here we consider a particular case of system \eqref{gKDVsystem}  which  couples two waves in a logarithmic way. More precisely, we consider
\be\label{logKDV}
\begin{cases}
	\partial_tu_{1} +\frac{1}{2}\partial_x\left(u_2\log u_1^2+u_2\log u_2^2\right)+\partial_x^3u_{1}=0,\\
	\partial_t	u_{2}+\frac{1}{2}\partial_x\left(u_1\log u_1^2+u_1\log u_2^2\right)+\partial_x^3u_{2}=0,
\end{cases}
\ee
The logarithmic KdV equation (log-KdV) 
$$
v_t+\partial_x^3v+\frac{1}{2}v\log v^2=0
$$
was rigorously justified in \cite{dp}. It appears in the context of Fermi–Pasta–Ulam  lattices with nonsmooth nonlinear Hertzian potentials (see also \cite{james} and \cite{nesterenko}). 
System \eqref{logKDV} may be seen as a particular case of \eqref{gKDVsystem} with $\mathcal{M}$ as in \eqref{partM} and
$$
R(u_1,u_2)=\frac{1}{2}u_1u_2\log (u_1^2u_2^2)-u_1u_2.
$$
As was pointed out in \cite{CNP}, in view of the log-Sobolev inequality, the energy $E$  is well-defined in $X=H_{per}^1([0,L])\times H_{per}^1([0,L])$.

We look for proportional periodic waves in the form $\phi=(\varphi(x-ct),\varphi(x-ct))$. Hence, \eqref{travwave} can be rewritten as
\begin{equation}\label{odelogkdvsystem1}
\begin{cases}
-\varphi''+c\varphi-\varphi\log\varphi^2+A_1=0,\\
-\varphi''+c\varphi-\varphi\log\varphi^2+A_2=0,
\end{cases}
\end{equation} 
Demanding $A_1=A_2=A$, \eqref{odelogkdvsystem1} reduces to
\begin{equation}\label{odelogkdvsystem}
-\varphi''+c\varphi-\varphi\log\varphi^2+A=0.
\end{equation} 
Equation \eqref{odelogkdvsystem} is exactly the same one obtained in \cite[Equation (1.2)]{CNP}. In particular the existence of periodic solutions for \eqref{odelogkdvsystem} may be obtained as in \cite{CNP}. For the convenience of the reader we recall the result below.

\begin{theorem}\label{teoexisanzero}
Define
	\begin{equation}\label{P1}
	\mathcal{P}_1=\{(c,A)\in\R^2; \;c\in\R, \; |A|<2e^{c/2-1}\}\nonumber
	\end{equation}
	and
	\begin{equation}\label{P3}\nonumber
	\mathcal{P}_2=\{(c,A)\in\R^2; \;c\in\R, \; |A|>2e^{c/2-1}\}.
	\end{equation}
	Let $(c_0,A_0)\in \mathcal{P}_1\cup\mathcal{P}_2$ be fixed.	Then there are an open neighborhood $\mathcal{O}$ of
	$(c_0,A_0)$ and a smooth family,
	$$
	(c,A)\in \mathcal{O}\mapsto \varphi_{(c,A)} \in H_{per}^2([0,L_0]),
	$$ of even $L_0$-periodic solutions of $(\ref{odelogkdvsystem})$,
	which  depends smoothly on $(c,A)\in \mathcal{O}$. In this case, $L_0$ is the period of $\varphi_{(c_0,A_0)}$ and it is determined by the level curves of the energy related to \eqref{odelogkdvsystem}.
	
	In addition, if $A_0>0$ is sufficiently large, taking a small $\mathcal{O}$ if necessary, we may assume that $\varphi_{(c,A)}>1$. 
\end{theorem}
\begin{proof}
The proof of this theorem was essentially given in \cite[Theorems 4.7 and 4.10]{CNP}. The only point left is that  $\varphi_{(c,A)}>1$ if $A_0$ is sufficiently large, which follows easily from a phase plane analysis. The phase spaces for $A=1$ and $A=2$ are shown in Figure1 \ref{figureA1A2} below.
\end{proof}

\begin{figure}[h!]
	\centering
	\includegraphics[scale=0.35]{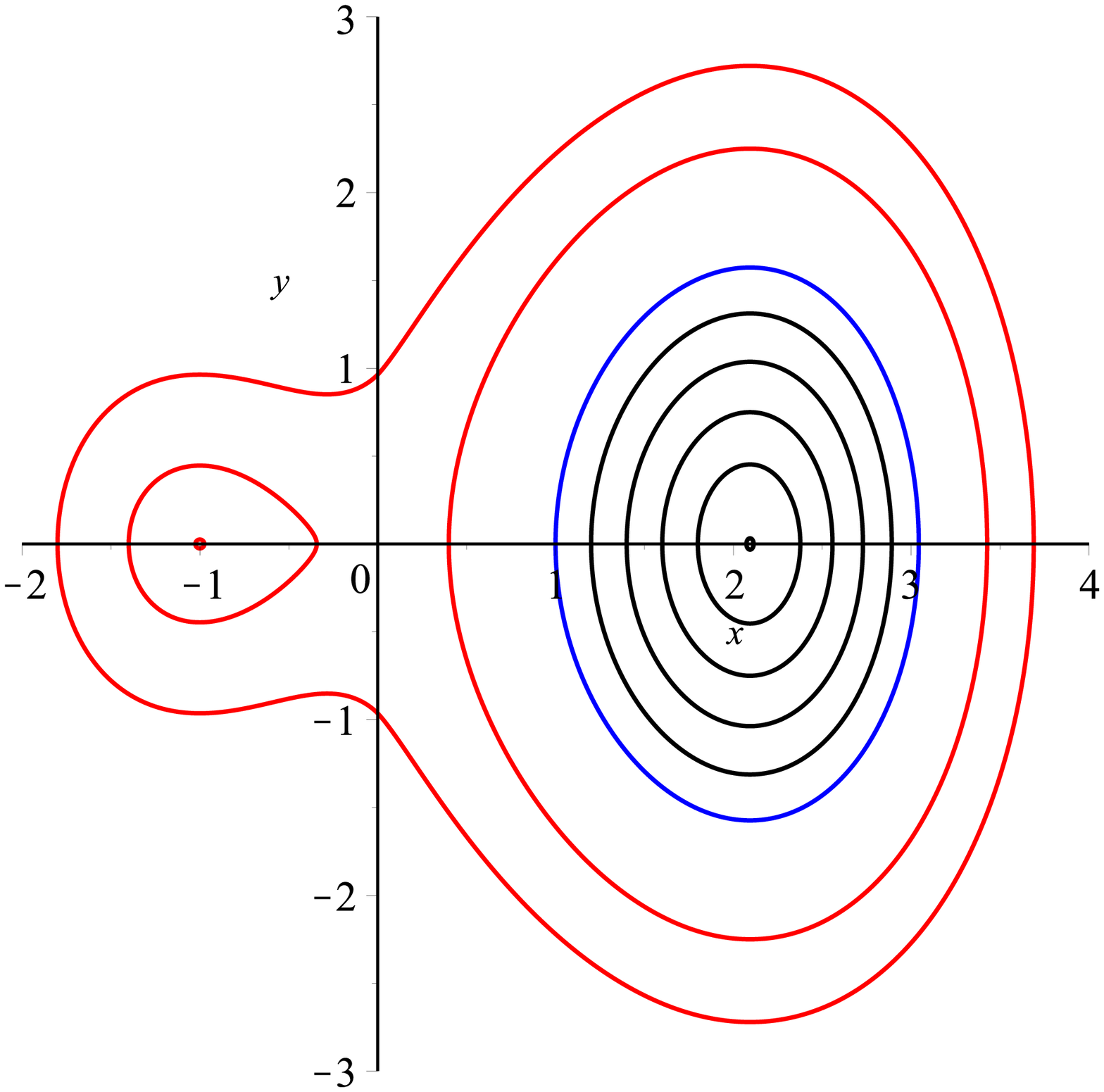} 
	\hspace{0.8cm}
	\includegraphics[scale=0.351]{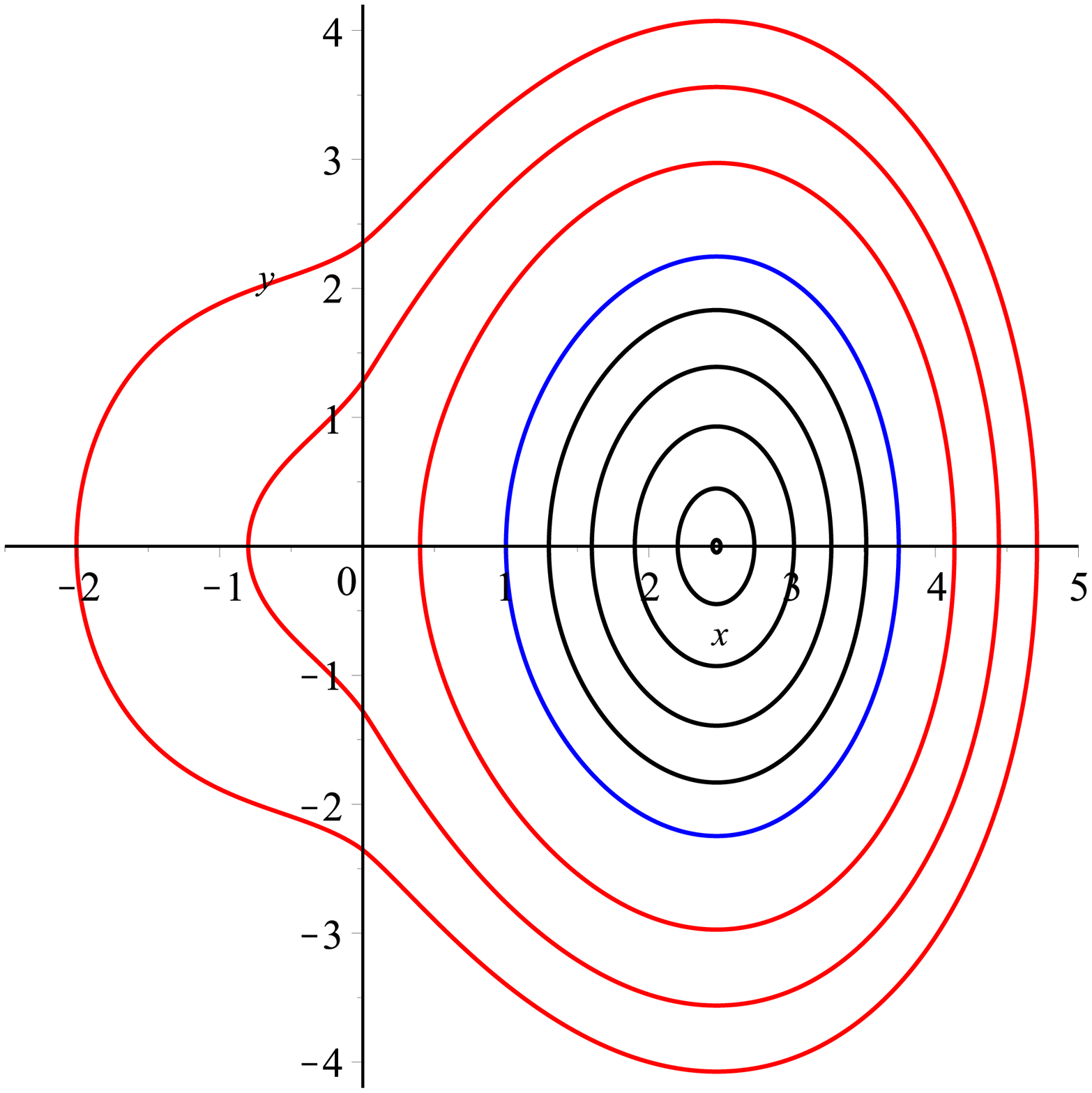}
	\caption{Left: Phase space for $c=1$ and $A=1$. Right: Phase space for $c=1$ and $A=2$. In both cases, the orbits in black are those for which $\varphi>1$.}
	\label{figureA1A2}
\end{figure}

We will therefore study the orbital stability of the solutions $(\varphi_{(c,A)},\varphi_{(c,A)})$, for $(c,A)$ in a neighborhood of $(c_0,A_0)$ with $A_0$ sufficiently large (fixed). To simplify notation we set $\varphi=\varphi_{(c,A)}$.
The linearized operator in this case has the form
\begin{equation}
\label{operatorlogkdv}
\mathcal{L}=
\left(-\partial_{x}^2+c\right)
 \begin{pmatrix}
1 & 0 \\
0 & 1
\end{pmatrix} -
 \begin{pmatrix}
1 & 1+\log \varphi^2 \\
1+\log \varphi^2 & 1
\end{pmatrix} 
\end{equation}

By setting $O$ to be the orthogonal  matrix
$${O}:=
 \begin{pmatrix}
1/\sqrt2 & 1/\sqrt2 \\
-1/\sqrt2 & 1/\sqrt2 
\end{pmatrix} ,$$
the operator $\mathcal{L}$ can be diagonalized. More precisely,
\begin{equation} \label{operator12}
O\mathcal{L}O^{-1}=
\begin{pmatrix}
\mathcal{L}_1 & 0 \\
0 & \mathcal{L}_2
\end{pmatrix}  
\end{equation}
where
$$\mathcal{L}_1:=-\partial_{x}^2+c-2-\log\varphi^2 \quad \mbox{and} \quad \mathcal{L}_2:=-\partial_{x}^2+c+\log \varphi^2. $$

Since $\varphi>1$ we have that $\mathcal{L}_2$ is a positive operator. It remains to study the spectral properties of $\mathcal{L}_1$. However, this has already been studied in  \cite{CNP}. More precisely, we have the following.

\begin{proposition}\label{oplogkdv}
For $(c,A)\in\mathcal{O}$, let
$\varphi_{(c,A)}$ be the $L_0$-periodic solution determined in Theorem \ref{teoexisanzero}. The closed, unbounded and self-adjoint
operator $\mathcal{L}_1$ has a
unique negative eigenvalue whose associated eigenfunction is even. Zero is  a simple eigenvalue with
associated eigenfunction $\varphi_{(c,A)}'$. Moreover, the rest of the
spectrum is bounded away from zero.
\end{proposition}
\begin{proof}
See \cite[Proposition 4.11]{CNP}.
\end{proof}

In view of the last proposition we conclude that $(H_1)$ holds. Next we will check that $(H_2)$ also holds. For $(c,A)\in\mathcal{O}$, we define
$$
\qquad\eta:=\frac{\partial}{\partial c}\varphi_{(c,A)},\ \qquad\beta:=\frac{\partial}{\partial A}\varphi_{(c,A)},
$$
and set
$$
M_{c}(\phi)=2\frac{\partial}{\partial c}\int_0^{L_0}\varphi_{(c,A)}(x)dx,\qquad  M_{A}(\phi)=2\frac{\partial}{\partial A}\int_0^{L_0}\varphi_{(c,A)}(x)dx,
$$
and
$$
F_{c}(\phi)=\frac{\partial}{\partial{c}}\int_0^{L_0}\varphi_{(c,A)}^2(x)dx, \qquad F_{A}(\phi)=\frac{\partial}{\partial{A}}\int_0^{L_0}\varphi_{(c,A)}^2(x)dx.
$$

The following proposition gives a simple criterion to obtain $(H_2)$.

\begin{proposition}\label{propK}
Let $K:\R^2\to\R$ be the function defined as
\begin{equation}\label{defK}	K(x,y)=x^2F_{c}(\phi)+xy(M_{c}(\phi)+F_{A}(\phi))+y^2M_{A}(\phi).
\end{equation}	
Assume that there is $(a,b)\in\R^2$ such that $K(a,b)>0$. Then, $(H_2) $ holds.
\end{proposition}
\begin{proof}
Let
$\Phi=(\Phi_1,\Phi_2)\in X$ and $Q(u)=-(aF(u)+bM(u))$, where $\Phi_1=\Phi_2:=a\eta+b\beta$. By taking the derivative with respect to $c$ and $A$ in \eqref{odelogkdvsystem} it is not difficult to see that
$$
\mathcal{L}\Phi=-\begin{pmatrix}
a\varphi+b\\
a\varphi+b
\end{pmatrix}
=Q'(\phi).
$$
Moreover, clearly $\langle \mathcal{L}\Phi,\phi'\rangle=0$ and
\[
\begin{split}
\langle\mathcal{L}\Phi,\Phi \rangle
&=-(a^2F_{c}(\phi)+abM_{c}(\phi)+abF_{A}(\phi)+b^2M_{A}(\phi))\\
&=-K(a,b).
\end{split}
\]
The proof is thus completed.
\end{proof}

We will use the above proposition to show that if $A_0>$  then $(H_2)$ holds.

\begin{corollary}\label{cor2logkdv}
If $A_0>0$ then $(H_2)$ holds.
\end{corollary}
\begin{proof}
First we differentiate \eqref{odelogkdvsystem} with respect to $c$, multiplying the obtained equation by $\varphi$ and integrating on $[0,L_0]$ we obtain that, after using integration by parts,
	\begin{equation}\label{difFom}
F_c(\phi)=F(\phi)-\frac{A}{2}M_c(\phi).
\end{equation}
Similarly, but now differentiating \eqref{odelogkdvsystem}
with respect to $A$, we deduce
\begin{equation}\label{difFA}
F_A(\phi)=\frac{1}{2}\left(M(\phi)-AM_A(\phi)\right).
\end{equation}
Differentiating \eqref{difFom} with respect to $A$,  \eqref{difFA} with respect to $c$, and comparing the result we see that $F_A(\phi)=M_c(\phi)$, which implies that function $K$ in Proposition \ref{propK} may be written as
\begin{equation}\label{newk}
K(x,y)=x^2F_c(\phi)+2xyM_c(\phi)+y^2M_A(\phi).
\end{equation}
Now using \eqref{difFom} and \eqref{difFA} in \eqref{newk}, we obtain
$$K(x,y)=\left(y^2-Axy+\frac{1}{4}A^2x^2\right)M_A(\phi)+\left(xy-\frac{1}{4}Ax^2\right)M(\phi)+x^2F(\phi).$$
By choosing $a\neq0$ and $b=\frac{1}{2}Aa$, we get
$$K(a,b)=a^2\left(\frac{1}{4}AM(\phi)+F(\phi)\right)>0.$$
From Proposition \ref{propK} we conclude the proof.
\end{proof}

 As an application of Theorem \ref{teoest} we may establish the following.

\begin{theorem}
Fix  $(c_0,A_0)\in \mathcal{P}_1\cup\mathcal{P}_2$ with $A_0$ sufficiently large. Let $L_0$ be the period of $\varphi_{(c_0,A_0)}$. Let $\phi=(\varphi_{(c,A)},\varphi_{(c,A)})$, where $\varphi_{(c,A)}$ is given in Theorem \ref{teoexisanzero}.
 Then $\phi$  is orbitally stable in $X=H_{per}^1([0,L_0])\times H_{per}^1([0,L_0])$ by the periodic flow of \eqref{logKDV}. 
\end{theorem}

\subsection{Liu-Kubota-Ko System}

In this last application, we consider the reduced Liu-Kubota-Ko system (see also \cite{absta})
\be\label{LKK}
\begin{cases}
	\partial_tu_{1} +\partial_x\big(\frac{1}{2}u_1^2-\mathcal{M}_{11}u_1-\mathcal{M}_{12}u_2\big)=0,\\
	\partial_tu_{2} +\partial_x\big(\frac{1}{2}u_2^2-\mathcal{M}_{12}u_1-\mathcal{M}_{11}u_2\big)=0, 
\end{cases}
\ee
where $\mathcal{M}_{11}$ and $\mathcal{M}_{12}$ have  symbols 
$$m_{11}(\kappa)=\kappa\coth(\kappa H)-\frac{1}{H}, \quad H>0,$$ 
and 
 $$m_{12}(\kappa)=-\frac{\kappa}{\sinh(\kappa H)}.$$
 Hence, \eqref{LKK} writes as \eqref{gKDVsystem} with 
 $$
 R(u_1,u_2)=\frac{1}{6}(u_1^3+u_2^3).
 $$

First of all we remark that $m(\kappa)$ satisfies \eqref{A1A2} only for sufficiently large values of $|\kappa|$. In order to get around this problem we note there is a constant  $\sigma_0\neq 0$ such that \eqref{A1A2} holds with $\sigma_0{\rm Id_2}+m(\kappa)$ instead of $m(\kappa)$  (see \cite[pages 201 and 204]{AL}). So, instead of $E$ we need to consider the new functional 
$$E_0(u)=\int_{0}^{L}\left\{ \frac{1}{2}\langle \mathcal{D} u,u\rangle_{\R^2}-R(u)\right\}dx,$$
 where $\mathcal{D}=\sigma_0{\rm Id_2}+\mathcal{M}$.

If we substitute the periodic wave $\phi=(\varphi(x-\omega t),\varphi(x-\omega t))$ into \eqref{LKK}, we see that both equations reduce to the same one, namely,
\begin{equation}\label{odelkk}
(\mathcal{M}_{11}+\mathcal{M}_{12})\varphi+\omega\varphi-\frac{1}{2}\varphi^2=0
\end{equation}
where $\omega=c+\sigma_0$, $c\in\R$, and the constant of integration is assumed to be zero. To obtain a solution for \eqref{odelkk}, we note that, at least formally, 
$$
\lim_{H\to+\infty}(m_{11}(\kappa)+m_{12}(\kappa))=|\kappa|.
$$
Thus, \eqref{odelkk} reduces to 
\begin{equation}\label{odebo}
\mathcal{H}\partial_x\varphi+\omega\varphi-\frac{1}{2}\varphi^2=0,
\end{equation}
where $\mathcal{H}$ stands for the Hilbert transform. The pseudo-differential equation in \eqref{odebo} is exactly the one which appears when one looks for traveling waves of the well known Benjamin-Ono equation (see, for instance, \cite{AN}). So our idea to get a solution for \eqref{odelkk} is to see it as a perturbation of \eqref{odebo}. To do so, 
we set $W=\frac{1}{H}$ and define for $W\geq0$,
\begin{equation}
\theta(\kappa,W):=
\begin{cases}
\kappa\coth(\kappa W^{-1})-W-{\displaystyle \frac{\kappa}{\sinh(\kappa W^{-1})}}, \quad \mbox{if} \quad W>0,\\
|k|,\quad \mbox{if} \quad W=0.
\end{cases}
\end{equation}

 Note that, for each fixed $\kappa\in\mathbb{Z}$, $\lim_{W\to 0^+ }{\theta}(\kappa,W)=|\kappa|$, which means that $\theta(k,\cdot)$ is a continuous function on $[0,\infty).$
In what follows, to give the explicit dependence of  $\mathcal{M}_{11}+\mathcal{M}_{12}$ with respect to $W$ we denote $\mathcal{M}_{11}+\mathcal{M}_{12}=(\mathcal{M}_{11}+\mathcal{M}_{12})_W$.

If $W=0$, our analysis above  combined with the results in \cite[page 1138]{AN} yields that equation \eqref{odelkk} admits the existence of a smooth curve of even $L$-periodic  solutions $\omega\in\left(\frac{2\pi}{L},\infty\right)\mapsto\varphi_{(\omega,0)}\in H_{per}^n[0,L]$,  $n\in\mathbb{N}$, with
\begin{equation}
\varphi_{(\omega,0)}(x)=\frac{4\pi}{L}\left(\frac{\sinh\left(\gamma\right)}{\cosh\left(\gamma\right)-\cos\left(\frac{2\pi x}{L}\right)}\right)
\end{equation}
and $\gamma:=\gamma(\omega)=\tanh^{-1}\left(\frac{2\pi}{\omega L}\right)$. The existence of solutions for $W>0$ small is given below.

\begin{theorem}\label{existsurface}
Let $\omega_0>\frac{2\pi}{L}$ be fixed. There exist an open neighborhood $\mathcal{O}\subset  (\frac{2\pi}{L},\infty)\times[0,\infty)$ containing $(\omega_0,0)$ and a smooth surface
\begin{equation*}
(\omega,W)\in \mathcal{O}\mapsto \varphi_{(\omega,W)} \in H_{per,e}^1([0,L]), 
\end{equation*}
of even $L$-periodic solutions of \eqref{odelkk}. 
\end{theorem}
\begin{proof}
We define $\Upsilon:(\frac{2\pi}{L},\infty)\times[0,\infty)\times H_{per,e}^{1}([0,L])\to L_{per,e}^2([0,L])$ by
\begin{equation}\label{Ups}
\Upsilon(\omega,W,f)=(\mathcal{M}_{11}+\mathcal{M}_{12})_W f+\omega f-\frac{1}{2}f^2,
\end{equation}
where $H_{per,e}^{1}([0,L])$  indicates the subspace of $H_{per}^{1}([0,L])$ constituted by even $L$-periodic functions. Since $\theta(\cdot,W)$ is an even function it follows that $(\mathcal{M}_{11}+\mathcal{M}_{12})_Wf$ is also even, for any  function $f$ in $H_{per,e}^{1}([0,L])$, implying that $\Upsilon$ is well defined.
	
Clearly, we have $\Upsilon(\omega_0,0,\varphi_{(\omega_0,0)})=0$. Moreover, note that $\Upsilon$ is smooth and its Fr\'echet derivative with respect to $f$ evaluated at $(\omega,W,\varphi)$ is 
$$\mathcal{G}_{(\omega,W)}=(\mathcal{M}_{11}+\mathcal{M}_{12})_{W}+\omega-\varphi.$$ 
Evaluating at $(\omega_0,0,\varphi_{(\omega_0,0)})$	we obtain (see Theorem 4.1 and page 1140 in \cite{AN}) that $\varphi_{(\omega_0,0)}'$ is an eigenfunction of  $\mathcal{G}_{(\omega_0,0)}$  associated to the eigenvalue zero.
Since $\varphi_{(\omega_0,0)}'$ does not belong to $H_{per,e}^{1}([0,L])$ (because it is odd), we conclude that $\mathcal{G}_{(\omega_0,0)}$ is one-to-one.  A simple analysis also shows that $\mathcal{G}_{(\omega_0,0)}$ is also surjective (see, for instance, \cite[Theorem 3.2]{CNP2}).  As consequence of the Implicit Function Theorem  we establish the desired result.
\end{proof}

\begin{remark}\label{inftyphi}
If $\varphi\in H_{per}^1([0,L])$ is a solution of $(\ref{odelkk})$ in the sense of distributions, then $\varphi \in H_{per}^n([0,L])$, for all $n\in\mathbb{N}$. This result can be proved by an iteration method and is a consequence of the fact that $\widehat{\varphi}\in \ell^1$ (see \cite[Proposition 3.1]{CNP2}). In particular, using the previous theorem, we have that $\varphi_{(\omega,W)}\to \varphi_{(\omega_0,0)}$, as $(\omega,W)\to (\omega_0,0)$, in $L^\infty_{per}([0,L])$.
\end{remark}

Next result shows that $\mathcal{G}_{(\omega,W)}$ has the same spectral property of $\mathcal{G}_{(\omega_0,0)}$  for $(\omega,W)$ sufficiently close to $(\omega_0,0)$.

\begin{proposition}\label{prop1}
Let $\varphi_{(\omega,W)}$ be the periodic traveling wave solution obtained in Theorem \ref{existsurface}. Then, there exists $\widetilde{\mathcal{O}}\subset\mathcal{O}$ such that, for all $(\omega,W)\in\widetilde{\mathcal{O}}$, operator $\mathcal{G}_{(\omega,W)}$ has only one negative eigenvalue which is simple and zero is a simple eigenvalue whose eigenfunction is $\phi_{(\omega,W)}'$.
\end{proposition}
\begin{proof}
Assume $(\omega,W)\in \mathcal{O}$ and consider the self-adjoint operador
$$
\mathcal{G}_{(\omega,W)}=(\mathcal{M}_{11}+\mathcal{M}_{12})_{W}+\omega-\varphi_{(\omega,W)}.
$$ 
defined on $L_{per}^2([0,L])$ with domain $D(\mathcal{G}_{(\omega,W)})=H_{per}^{\frac{1}{2}}([0,L])$.
	
Let us show that ${\mathcal{G}}_{(\omega,W)}$ converges to ${\mathcal{G}}_{(\omega_0,0)}$, as $(\omega,W)\to (\omega_0,0)$, in the metric  \textit{gap} $\widehat{\delta}$ (see Sections 2 and 3 of Chapter IV in \cite{kato1}). First we write $\mathcal{G}_{(\omega,W)}=(\mathcal{M}_{11}+\mathcal{M}_{12})_{W}+N_{(\omega,W)}$, where $N_{(\omega,W)}f=(\omega-\varphi_{(\omega,W)})f$. Next, we remark that $(\mathcal{M}_{11}+\mathcal{M}_{12})_{W}-(\mathcal{M}_{11}+\mathcal{M}_{12})_{0}$ and $N_{(\omega,W)}-N_{(\omega_0,0)}$ are bounded operators on $L^2_{per}([0,L])$ with operator norms tending to $0$ as $(\omega,W)\to (\omega_0,0)$. In fact, one has
$$\|(\mathcal{M}_{11}+\mathcal{M}_{12})_{W}-(\mathcal{M}_{11}+\mathcal{M}_{12})_{0}\|=|\theta(k,W)-\theta(k,0)|_{\infty},$$
so, by the definition of $\theta$, it follows that 
\begin{equation}\label{limMS}
\lim_{W\to 0^+ }\|(\mathcal{M}_{11}+\mathcal{M}_{12})_{W}-(\mathcal{M}_{11}+\mathcal{M}_{12})_{0}\|=0
\end{equation}
Also, since 
$$\|N_{(\omega,W)}-N_{(\omega_0,0)}\|\leq |\omega-\omega_0|+\|\varphi_{(\omega,W)}-\varphi_{(\omega_0,0)}\|_{L^\infty_{per}([0,L])},$$
it follows, by Remark \ref{inftyphi}, that
\begin{equation}\label{limN}
\lim_{(\omega,W)\to (\omega_0,0) }\|N_{(\omega,W)}-N_{(\omega_0,0)}\|=0.
\end{equation}

Finally, Theorem 2.17 in \cite[Chapter IV]{kato1} implies that
\begin{eqnarray}
\widehat{\delta}(\mathcal{G}_{(\omega,W)},\mathcal{G}_{(\omega_0,0)})&=&\widehat{\delta}((\mathcal{M}_{11}+\mathcal{M}_{12})_{W}+N_{(\omega,W)},(\mathcal{M}_{11}+\mathcal{M}_{12})_{0}+N_{(\omega_0,0)}) \nonumber \\
&\leq&2\left(1+\|N_{(\omega,W)}\|\right)\times\widehat{\delta}((\mathcal{M}_{11}+\mathcal{M}_{12})_{W},(\mathcal{M}_{11}+\mathcal{M}_{12})_{0}\\
&&+N_{(\omega_0,0)}-N_{(\omega,W)}) \nonumber \\
&\leq&2\left(1+\|N_{(\omega,W)}\|\right)\times\left[\widehat{\delta}((\mathcal{M}_{11}+\mathcal{M}_{12})_{W},(\mathcal{M}_{11}+\mathcal{M}_{12})_{0}) \right. \nonumber \\
&&+\left. \widehat{\delta}((\mathcal{M}_{11}+\mathcal{M}_{12})_{0},(\mathcal{M}_{11}+\mathcal{M}_{12})_{0}+N_{(\omega_0,0)}-N_{(\omega,W)})\right] \nonumber \\
&\leq&2\left(1+\|N_{(\omega,W)}\|\right)\times\Big\{\|(\mathcal{M}_{11}+\mathcal{M}_{12})_{W}-(\mathcal{M}_{11}+\mathcal{M}_{12})_{0}\| \nonumber \\
&&+ \|N_{(\omega,W)}-N_{(\omega_0,0)}\|\Big\} \nonumber
\end{eqnarray}
Using \eqref{limMS} and \eqref{limN}, we conclude that $\widehat{\delta}(\mathcal{G}_{(\omega,W)},\mathcal{G}_{(\omega_0,0)})\to 0$, as $(\omega,W)\to (\omega_0,0)$. 
Consequently, by taking into account that zero is an eigenvalue of $\mathcal{G}_{(\omega,W)}$ with eigenfunction $\varphi_{(\omega,W)}'$, from Theorem 3.16  in \cite[Chapter IV]{kato1}, we conclude that  there exist a neighborhood $\widetilde{\mathcal{O}}$ of $(\omega_0,0)$ where $\mathcal{G}_{(\omega,W)}$ has the same spectral properties of $\mathcal{G}_{(\omega_0,0)}$, which is to say that it has only one negative eigenvalue which is simple and zero is a simple eigenvalue (see Theorem 4.1 and page 1140 in \cite{AN}).  
\end{proof}

From the above results we conclude that $H_1$ holds. 
Next, we are going to prove that $H_2$ holds. Since 
\begin{equation} \label{operatorlkk}
\mathcal{L}= 
\left( \begin{array}{cc}
\mathcal{G}_{(\omega,W)} & 0 \\
0 &\mathcal{G}_{(\omega,W)}
\end{array} \right),
\end{equation}
it follows that, considering $\Phi=(\frac{\partial}{\partial \omega}\varphi,\frac{\partial}{\partial \omega}\varphi)$, we get
\begin{equation}
I=\langle\mathcal{L}\Phi,\Phi\rangle =-2\frac{\partial}{\partial \omega}\|\varphi_{(\omega,W)}\|^2_{L^2_{per}}.
\end{equation} 

Recall that Angulo and Natali proved in \cite[page 1139]{AN} that 
$$\frac{\partial}{\partial \omega}\|\varphi_{(\omega_0,0)}\|^2_{L^2_{per}}>0$$
for each $\omega_0\in\left(\frac{2\pi}{L},\infty\right)$. Since the inequality is strict, it must also hold for values of $W$ which are sufficiently near of $0$. Theorem \ref{teoest} implies that $\varphi_{(\omega,W)}$ is orbitally stable in  
$X=H_{per}^{\frac{1}{2}}([0,L])\times H_{per}^\frac{1}{2}([0,L])$ by the periodic flow of \eqref{LKK}, provided that $W$ is sufficiently small.

\section*{Acknowledgements}

F. C. is supported by FAPESP/Brazil grant 2017/20760-0. A. P. is partially supported by CNPq/Brazil grants 402849/2016-7 and 303098/2016-3.

\end{document}